\documentclass{article}
\usepackage{tikz}
\usetikzlibrary{matrix}

\usepackage{amssymb}
\usepackage{verbatim}
\usepackage{array}
\usepackage{latexsym}
\usepackage{enumerate}
\usepackage{amsmath}
\usepackage{amsfonts}
\usepackage{amsthm}
\usepackage{color}
\usepackage[english]{babel}

\newtheorem{theorem}{Theorem}[section]
\newtheorem{proposition}[theorem]{Proposition}
\newtheorem{lemma}[theorem]{Lemma}
\newtheorem{corollary}[theorem]{Corollary}

\newtheorem{rem}[theorem]{Remark}
\def\cC{\mathcal C}

\def\cP{\mathcal P}

\def\cX{\mathcal X}

\def\K{\mathbb{K}}

\def\ord{\mbox{\rm ord}}

\def\div{\mbox{\rm div}}

\def\Div{\mbox{\rm div}}

\def\supp{\mbox{\rm Supp}}

\def\gg{\mathfrak{g}}
\def\gm{\mathfrak{\gamma}}
\newcommand{\PSL}{\mbox{\rm PSL}}
\newcommand{\PGL}{\mbox{\rm PGL}}

\newcommand{\PSU}{\mbox{\rm PSU}}

\newcommand{\aut}{\mbox{\rm Aut}}

\def\supp{{\rm Supp}}

\newcommand{\ha}{{\textstyle\frac{1}{2}}}

\newcommand{\bA}{{\bf A}}

\newcommand{\bS}{{\bf S}}

\title{Transcendency Degree One Function Fields Over a Finite Field with Many Automorphisms}
\date{}
\author{G\'abor Korchm\'aros, Maria Montanucci and Pietro Speziali}
\begin{document}
\maketitle

%\thanks{{\em Keywords}: Algebraic curves, Positive characteristic, Automorphism groups.}

\begin{abstract} Let $\mathbb{K}$ be the algebraic closure of a finite field $\mathbb{F}_q$ of odd characteristic $p$. For a positive integer $m$ prime to $p$, let $F=\mathbb{K}(x,y)$ be the transcendency degree $1$ function field defined by  $y^q+y=x^m+x^{-m}$. Let $t=x^{m(q-1)}$ and $H=\mathbb{K}(t)$. The extension $F|H$ is a non-Galois extension. Let $K$ be the Galois closure of $F$ with respect to $H$. By Stichtenoth \cite{SH}, $K$ has genus $\gg(K)=(qm-1)(q-1)$, $p$-rank (Hasse-Witt invariant) $\gamma(K)=(q-1)^2$ and a $\mathbb{K}$-automorphism group of order at least $2q^2m(q-1)$. In this paper we prove that this subgroup is the full $\mathbb{K}$-automorphism group of $K$; more precisely $\aut_{\mathbb {K}}(K)=Q\rtimes D$ where $Q$ is an elementary abelian $p$-group of order $q^2$ and $D$ has a index $2$ cyclic subgroup of order $m(q-1)$. In particular, $\sqrt{m}|\aut_{\mathbb{K}}(K)|>\gg(K)^{3/2}$, and if $K$ is ordinary (i.e. $\gg(K)=\gamma(K)$) then $|\aut_{\mathbb{K}}(K)|>\gg^{3/2}$. On the other hand, if $G$ is a solvable subgroup of the $\mathbb{K}$-automorphism group of an ordinary, transcendency degree $1$ function field $L$ of genus $\gg(L)\geq 2$ defined over $\mathbb{K}$,
then $|\aut_{\mathbb{K}}(K)|\le 34 (\gg(L)+1)^{3/2}<68\sqrt{2}\gg(L)^{3/2}$; see \cite{KM}. This shows that $K$ hits this bound up to the constant $68\sqrt{2}$.

Since $\aut_{\mathbb{K}}(K)$ has several subgroups, the fixed subfield $F^N$ of such a subgroup $N$ may happen to have many automorphisms provided that the normalizer of $N$ in $\aut_{\mathbb{K}}(K)$ is large enough.
%as the quotient group of $\aut_{\mathbb{K}}(K)/N$ is a subgroup of $\aut_{\mathbb{K}}(F^N)$.
This possibility is worked out for subgroups of $Q$.
\end{abstract}
\section{Introduction}
\label{intr}
Let $L$ be a transcendency degree one function field defined over an algebraically closed field $\mathbb{K}$, i.e. $L=\mathbb{K}(\cX)$ where $\cX$ is an algebraic curve defined over $\mathbb{K}$. It is well known that if $L$ is neither rational, nor elliptic then the $\mathbb{K}$-automorphism group $\aut(L)$ of $L$ is finite. More precisely, $|\aut(L)|\leq 16\gg(L)^4$ with just one exception, namely the Hermitian function field $H=H(x,y),\,y^q+y=x^{q+1}$ with $q=p^k$ whose genus equals $\ha q(q-1)$ and $\mathbb{K}$-automorphism group has order $(q^3+1)q^3(q^2-1)$; see \cite{stichtenoth1973II}. This bound was refined by Henn in \cite{henn1978} \textcolor{blue} and for special families of curves in \cite{GK3,GK4,GK5,gupr}.

In \cite{KM} the authors investigated the case where $L$ is ordinary, i.e. its genus and $p$-rank coincide, and they
showed for this case that if $G$ is a solvable subgroup of $\aut(L)$ then
\begin{equation}
\label{27oct2016}
|G|\le 34 (\gg(L)+1)^{3/2}<68\sqrt{2}\gg(L)^{3/2}.
\end{equation}

By Stichtenoth \cite{SH}, the Galois closure $K$ of $F|H$ where $F=\mathbb{K}(x,y)$ with $y^q+y=x^m+x^{-m}$ where $q=p^k$, $m$ is a positive integer prime to $p$, $H=\mathbb{K}(x^{m(q-1)})$, has genus $\gg(K)=(q-1)(qm-1)$, $p$-rank $\gamma(K)=(q-1)^2$ and size of the Galois group ${|\rm{Gal}}(K|H)|\geq q^2m(q-1)$. For $m=1$, $K$ is ordinary and it provides an example hitting the bound (\ref{27oct2016}), up to the constant term.

In Section \ref{mains} we prove that this subgroup is almost the full $\mathbb{K}$-automorphism group of $K$; more precisely $\aut_{\mathbb {K}}(K)=Q\rtimes D$ where $Q$ is an elementary abelian $p$-group of order $q^2$ and $D$ has a index $2$ cyclic subgroup of order $m(q-1)$. Moreover, $Q$ is defined over $\mathbb{F}_{q^2}$ while $D$ is defined over $\mathbb{F}_{q^r}$ where $r$ is the smallest positive integer such that $m(q-1) \mid (q^r-1)$. We also give an explicit representation for $K$ showing that $K=\mathbb{K}(x,y,z)$ with $y^q+y=x^m+x^{-m}$ and $z^q+z=x^m$.

Since $\aut_{\mathbb{K}}(K)$ has several subgroups, the fixed subfield $F^N$ of %such a subgroup $N$
some of such subgroups $N$ may happen to have many automorphisms provided that the normalizer of $N$ in $\aut_{\mathbb{K}}(K)$ is large enough.
%as the quotient group of $\aut_{\mathbb{K}}(K)/N$ is a subgroup of $\aut_{\mathbb{K}}(F^N)$.
In Section \ref{quot}, this possibility is worked out for subgroups of $\Delta$.

\section{Background and Preliminary Results}\label{sec2}
In this paper, $\mathbb{K}$ denotes an algebraically closed field of odd characteristic $p$. Let $L$ denote a transcendency degree 1 function field with constant field $\mathbb{K}$; equivalently let $L$ denote the function field $\mathbb{K}(\cX)$ of a (projective, non-singular, geometrically irreducible, algebraic) curve $\cX$ defined over $\mathbb{K}$. The subject of our paper is the group of automorphisms $\aut_{\K}(L)$ of $L$ which fix $\mathbb{K}$ elementwise, and we begin by collecting basic facts and known results on $\aut_{\K}(L)$ that will be used in our proofs. For more details, the \textcolor{blue}{r}eader is referred to \cite{HKT} and \cite{stichtenoth1993}.

For a subgroup $G$ of $\aut_\mathbb{K}(L)$, the fixed field $L^G$ of $L$ is the subfield of $L$ fixed by every element in $G$. The field extension $L|L^G$ is  Galois of degree $|G|$. Take a place $\bar{P}$ of $L^G$ together with a place $P$ of $L$ lying over $P$, that is, let $P$ be an extension of $\bar{P}$ to $L$. The integer $e=e(P|\bar{P})$ defined by $v_{P}(x)=e v_{\bar{P}}(x)$ for all $x \in L^G$ is the \emph{ramification index} of $P|\bar{P}$, and  $P|\bar{P}$ is \emph{unramified} if $e(P|\bar{P})=1$, otherwise it is \emph{ramified}.  If $P|\bar{P}$ is ramified then is either \emph{wildly} or \emph{tamely} ramified according as $p$ divides $e(P|\bar{P})$ or not. Furthermore,  $\bar{P}$ is \emph{ramified in} $L|L^G$ if $P|\bar{P}$ is ramified for at least one place $P$ of $L$, otherwise $\bar{P}$ is \emph{unramified in} $L|L^G$, and the adjective wild or tame is used for $\bar{P}$ according as at least one or none of the places $P$ of $L$ lying over $\bar{P}$ is wild or tame.
Also, a place $\bar{P}$ of $L^G$ is \emph{totally ramified} in $L|L^G$ if there is just one extension $P$ of $\bar{P}$ in $L$, and if this occurs then $e(P|\bar{P})=|G|$. Moreover, $L|L^G$ is an \emph{unramified extension} if no extension of $\bar{P}$ to $L$ is ramified; otherwise $L|L^G$ is an \emph{unramified extension}. If each extension  $P|\bar{P}$ is tame then $L|L^G$ is a \emph{tame Galois extension}; otherwise it is a \emph{wild Galois extension}.

On the set $\mathcal{P}$ of all places of $L$, $G$ has a faithful action. For $P\in \mathcal{P}$, the \emph{stabilizer} $G_P$ of $P$ in $G$ is the subgroup of $G$ consisting of all elements of $G$ fixing $P$. A necessary and sufficient condition for a place $P\in \mathcal{P}$ to be ramified is $|G_P|>1$, the ramification index $e_P$ being equal to $|G_P|$. The \emph{$G$-orbit} of $P\in \mathcal{P}$ consists of the images of $P$ under the action of $G$ on $\mathcal{P}$, and it is a \emph{long} or \emph{short} orbit according as $G_P$ is trivial or not. If $o$ is a $G$-orbit then $|o|=|G|/|G_P|$ for any
place $P\in o$. If no $G$-orbit is short then no nontrivial element in $G$ fixes a place in $\cP$, that is, $L|L^G$ is an unramified Galois extension, and the converse also holds.

Assume now that $L$ is neither rational nor elliptic. Then $L$ has genus $\gg(L)\geq 2$, and $G$ is finite with a finite number of short orbits on $\mathcal{P}$. For an integer $i\geq -1$, the $i$-th ramification group $G_P^{(i)}$ of the extension $P|\bar{P}$ is defined to be
$${\mbox{$G_P^{(i)}=\{g \in G \mid \ord_P(g(z)-z)\geq i+1,$ for  all $z \in O_{P} \}$,}} $$
where $O_{P}$ is the local ring at $P$ in $L$. These ramification groups are normal subgroups of $G_P$ and they form a decreasing chain $G_P=G_P^{(0)}\geq G_P^{(1)}\geq \cdots \geq \{1\}$. Here $G_P^{(0)}=G_P$ whereas  $G_P^{(1)}$ is the (unique) Sylow $p$-subgroup of $G_P,$ and $G_P=G_P^{(1)}\rtimes C$ where
the complement $C$ in the semidirect product $G_P^{(1)}\rtimes C$ is cyclic.
The Hurwitz genus formula states that
\begin{equation}
    \label{eq1}
2\gg(L)-2=|G|(2\gg(L^G)-2)+\sum_{P\in \mathbb{P}_{L}} d_P.
    \end{equation}
where $\gg(L^G)$ is the genus of $L^G$, and
\begin{equation}
\label{eq1bis}
d_P= \sum_{i\geq 0}(|G_P^{(i)}|-1).
\end{equation}

Let $\gamma(L)$ denote the $p$-rank (equivalently, the Hasse-Witt invariant of $L$). If $S$ is a $p$-subgroup of $\aut_{\mathbb{K}}(L)$ then
the Deuring-Shafarevich formula, see \cite{sullivan1975} or \cite[Theorem 11,62]{HKT}, states that
\begin{equation}
    \label{eq2deuring}
    %modifica 23 marzo
\gamma-1={|S|}(\bar{\gamma}-1)+\sum_{i=1}^k (|S|-\ell_i),
    \end{equation}
    %%modifica 2 marzo 2009
where $\gamma(L^S)$ is the $p$-rank of $L^S$ and $\ell_1,\ldots,\ell_k$ denote the sizes of the short orbits of $S$. Both the Hurwitz and Deuring-Shafarevich formulas hold true for rational and elliptic curves provided that $G$ is a finite subgroup.

A subgroup of $\aut_\mathbb{K}(L)$ is a \emph{$p'$-group} (or \emph{a prime to $p$}) group if its order is prime to $p$. A subgroup $G$ of $\aut_{\mathbb{K}}(L)$ is \emph{tame} if the $1$-point stabilizer of any point in $G$ is $p'$-group. Otherwise, $G$ is \emph{non-tame} (or \emph{wild}). Every $p'$-subgroup of $\aut_{\mathbb{K}}(L)$ is tame, but the converse is not always true. If $G$ is tame then the classical Hurwitz bound $|G|\leq 84(\gg(L)-1)$ holds, but for non-tame groups this is far from being true. The Stichtenoth bound $|G|\leq 16 \gg(L)^4$ holds for any $L$ with $\gg(L)\geq 2$  other than the Hermitian function field.

From Group Theory, we use the following three deep results, see \cite{mv1982,gw1,gw2}.

\begin{lemma}[Dickson's classification of finite subgroups of the projective linear group $\PGL(2,\mathbb{K})$]
\label{lem30oct2016}
The finite subgroups of the group $\PGL(2,\mathbb{K})$ are isomorphic to one of the following groups:
\begin{enumerate}
\item[\rm(i)] prime to $p$ cyclic groups;
\item[\rm(ii)] elementary abelian $p$-groups;
\item[\rm(iii)] prime to $p$ dihedral groups;
\item[\rm(iv)] the alternating group $\bA_4$;
\item[\rm(v)] the symmetric group $\bS_4$;
\item[\rm(vi)] the alternating group $\bA_5$;
\item[\rm(vii)] the semidirect product of an elementary abelian
$p$-group of order $p^h$ by a cyclic group of order $n>1$ with
 $n\mid(q-1);$
\item[\rm(viii)] $\PSL(2,p^f)$;
\item[\rm(ix)] $\PGL(2,p^f)$.
\end{enumerate}
\end{lemma}
\begin{lemma}[Feith-Thompson theorem]
\label{lemma30Aoct2016} Every finite group of odd order is solvable.
\end{lemma}
\begin{lemma}[Alperin-Gorenstein-Walter theorem]
\label{lemma30oct2016} If $\Gamma$ is a finite simple group of $2$-rank two (i.e. $\Gamma$ contains no elementary abelian subgroup of order $8$), then one of the following holds:
\begin{itemize}
\item[\rm(i)] The Sylow $2$-subgroups of $\Gamma$ are dihedral, and $\Gamma$ is isomorphic to either $\PSL(2,n)$ with an odd prime power $n\ge 5$, or to the alternating group $\bA_7$.
\item[\rm(ii)] The Sylow $2$-subgroups of $\Gamma$ are semi-dihedral and $\Gamma$ is isomorphic to either $\PSL(3,n)$ with an odd prime power $n\equiv -1 \pmod 4$, or to $\PSU(3,n), n\equiv 1 \pmod 4$, or to the Mathieu group $M_{11}$.
\item[\rm(iii)] The Sylow $2$-subgroups of $\Gamma$ are wreathed,  and $\Gamma$ is isomorphic to either $\PSL(3,n)$ with an odd prime power $n\equiv 1 \pmod 4$, or to $\PSU(3,n), n\equiv -1 \pmod 4$, or to $\PSU(3,4)$.
\item[\rm(iv)] $\Gamma$ isomorphic to $\PSU(3,4)$.
\end{itemize}
\end{lemma}
From now on, $\mathbb{K}$ is the algebraic closure of a finite field $\mathbb{F}_q$ of odd order $q=p^h$ with $h\geq 1$, $m\geq 1$ is an integer prime to $p$,  $F=\mathbb{K}(x,y)$ is the transcendency degree $1$ function field defined by  $y^q+y=x^m+x^{-m}$, $t=x^{m(q-1)}$ and $H$ is the rational subfield $\mathbb{K}(t)$ of $F$.

\section{Galois closure of $F|H$}
Let $F$ and $H$ be as defined in Section \ref{intr}.
Our first step is to give an explicit presentation of the Galois closure of $F|H$.
\begin{proposition}
\label{pro1}
The Galois closure of $F|H$ is $\mathbb{K}(x,y,z)$ with
\begin{equation}
\label{eq2} y^q+y=x^m+\frac{1}{x^m},
\end{equation}
\begin{equation}
\label{eq3} z^q+z=x^m.
\end{equation}
\end{proposition}
\begin{proof} Let $K$ denote the function field $\mathbb{K}(x,y,z)$ given by (\ref{eq2}) and (\ref{eq3}). We show first that $K$ contains a subfield isomorphic to an Artin-Mumford function field. For this, let $s=z-y$. Then (\ref{eq2}) reads
$$s^q+s=y^q+y-(z^q+z)=\frac{1}{x^m},$$ whence by (\ref{eq3})
\begin{equation}
\label{eq4} s^q+s=\frac{1}{z^q+z}.
\end{equation}
The function field $L=\mathbb{K}(x,s,z)$ with (\ref{eq3}) and (\ref{eq4}) is a subfield of $K$. Actually, $K=L$ as $y=z-s$, and $AM=\mathbb{K}(s,z)$ with (\ref{eq4}) is an Artin-Mumford subfield of $K$. Also,
%$$\dim [L:H]=\dim[K:H]=\dim [K:F]\,\dim[F:H]=q^2m(q-1).$$
$$[L:H]=[K:H]= [K:F]\,[F:H]=q^2m(q-1).$$
It remains to show that $\aut(L)$ has a subgroup of order $q^2m(q-1)$ fixing $t$. Take a positive integer $r$ for which $m|(q^r-1)$. Let $\mathcal{V}$ be the subgroup of $\mathbb{F}_{q^r}^*$ consisting of all elements $v$ such that $v^m\in \mathbb{F}_q^*$. Obviously, $\mathcal{V}$ is a cyclic group of order $(q-1)m$.

For $\alpha,\beta \in \mathbb{F}_{q^2}$ with ${\rm{Tr}}(\alpha)=\alpha^q+\alpha=0,\,{\rm{Tr}}(\beta)=\beta^q+\beta=0$, and $v\in \mathcal{V}$, let $\varphi_ {\alpha,\beta,v}(x,s,z)$ denote the $\mathbb{K}$-automorphism of $K$
\begin{equation}
\varphi_ {\alpha,\beta,v}(x,s,z)=(vx,v^{-m}s+\alpha,v^m z+\beta).
\end{equation}
Then $\varphi_ {\alpha,\beta,v}(s)^q+\varphi_ {\alpha,\beta,v}(s)=v^{-m}(s^q+s)$, and $\varphi_ {\alpha,\beta,v}(z)^q+\varphi_ {\alpha,\beta,v}(z)=v^m(z^q+z)$. This shows that
(\ref{eq4}) is left invariant by $\varphi_ {\alpha,\beta,v}(x,s,z)$. Furthermore, $\varphi_ {\alpha,\beta,v}(x)^m=v^m x^m.$ Let $$\Phi:=\{\varphi_ {\alpha,\beta,v}|\,v\in \mathcal{V}, \alpha^q+\alpha=0,\beta^q+\beta=0\}.$$ A straightforward computation shows that
%$V :=\{\varphi_{0,0,v} \mid v \in \mathcal{V} \}$
$$\varphi_{\alpha,\beta,v}\circ\varphi_{\alpha',\beta',v'}=\varphi_{v^{-m}\alpha'+\alpha,v^m\beta'+\beta,vv'}.$$
and hence $\Phi$ is a subgroup of $\aut_\mathbb{K}(L)$ of order $q^2m(q-1)$. Furthermore,
$$\varphi_{\alpha,\beta,v}(t)=\varphi_{\alpha,\beta,v}(x^{m(q-1)})={((\varphi_{\alpha,\beta,v}(x))^m)}^{q-1}=v^{m(q-1)}x^{m(q-1)}=t.$$
Since %$\dim[L:H]=\dim[K:H]=q^2m(q-1)$,
$[L:H]=[K:H]=q^2m(q-1)$, the claim follows.
\end{proof}
Our proof of Proposition \ref{pro1} also gives the following result.
\begin{lemma}
\label{lem16Aluglio2016}
The Galois group of the Galois closure $K$ of $F|H$ is $\Phi$.
\end{lemma}
\section{Some subgroups of $ \aut_{\mathbb{K}}(K)$} From Lemma \ref{lem16Aluglio2016},
$\Phi$ is a subgroup of $\aut_{\mathbb{K}}(K)$ of order $q^2m(q-1)$.
Actually, $\aut_{\mathbb{K}}(K)$ is larger than $\Phi$.
\begin{lemma}
\label{lem16Bluglio2016}
$|\aut_{\mathbb{K}}(K)| \geq 2q^2m(q-1)$.
\end{lemma}
\begin{proof} Let
$$ \xi:\,(x,s,z)\mapsto \bigg(\frac{1}{x},z,s \bigg).$$
By a straightforward computation, $\xi$ $\in \aut_{\mathbb{K}}(F)$, and $\xi\not\in \Phi$ is an involution. Since $\xi \varphi_{\alpha,\beta,v} \xi = \varphi_{\beta,\alpha,v^{-1}}$ for every $\varphi_{\alpha,\beta,v} \in \Phi$, the normalizer of $\Phi$ contains $\xi$. Thus, $|\langle \Phi, \xi \rangle|=2q^2m(q-1)$ by Lemma \ref{lem16Aluglio2016}.
\end{proof}
From the proof of Lemma \ref{lem16Bluglio2016}, $G=\Phi\rtimes \langle \xi \rangle$ is a subgroup of $ \aut_{\mathbb{K}}(K)$. Our main goal is to prove that $G=\aut_{\mathbb{K}}(K)$. The proof needs several results on the structure of $\aut_{\mathbb{K}}(K)$ which are stated and proven below. For this purpose, the following subgroups of $\aut_{\mathbb{K}}(K)$ are useful.
\begin{enumerate}
\item[(i)] $\Psi:=\{\varphi_{\alpha,\alpha,1}|\alpha^q+\alpha=0\}$ of order $q$.
\item[(ii)] $\Delta := \{ \varphi_{\alpha,\beta,1} \mid \alpha^q+\alpha=\beta^q+\beta=0\}$ of order $q^2$.
\item[(iii)] $W :=\{\varphi_{0,0,v} \mid v^m =1 \}$ of order $m$.
\item[(iv)] $V :=\{\varphi_{0,0,v} \mid v \in \mathcal{V}\} $ of order $(q-1)m$.
\item[(v)] $M:=\{\varphi_{\alpha,\beta,v} \in \Phi \mid v^m=1 \}$.
\end{enumerate}
Obviously, both $\Delta$ and $\Psi$ are elementary abelian $p$-groups while both $V$ and $W$ are prime to $p$ cyclic groups.
\begin{proposition}
\label{pro2}  $K|F$ is an unramified Galois extension of degree $q$. Furthermore, $\gg(K) = (q-1)(qm-1)$ and $\gm(K)=(q-1)^2$.
\end{proposition}
\begin{proof} We show that $F=K^{\Psi}$. From $\varphi_{\alpha,\alpha,1}(x,s,z)=(x,s+\alpha,z+\alpha)$,   $$\varphi_{\alpha,\alpha,1}(y)=\varphi_{\alpha,\alpha,1}(z-s)=\varphi_{\alpha,\alpha,1}(z)-\varphi_{\alpha,\alpha,1}(s)=
z+\alpha-(s+\alpha)=z-s=y.$$ Moreover, $\varphi_{\alpha,\alpha,1}(x)=x$. Therefore, $K^\Psi$ contains $F$. Since %$\dim [K:F]=q$,
$[K:F]=q$ this yields $F=K^\Psi$
whence the first claim follows. We show that no nontrivial element in $\Psi$ fixes a place of $K$. From the definition of $\Psi$, every $\psi\in \Psi$ leaves the Artin-Mumford subfield $AM=\mathbb{K}(s,z)$ invariant. By a straightforward computation, if $\psi$ is nontrivial, then it fixes no place of $AM$. But then $\psi$ fixes no place of $L$, and hence $K|F$ is unramified. Therefore, the Hurwitz genus formula and the Deuring-Shafarevich formula yield the second claim.
\end{proof}
Proposition \ref{pro2} has the following corollary.
\begin{corollary}
\label{cor1nov2016} A necessary and sufficient condition for $F$ to be ordinary, i.e. $\gg(F)=\gamma(F)$, is $m=1$.
\end{corollary}
%Let $\Delta := \{ \varphi_{\alpha,\beta,1} \mid \alpha^q+\alpha=\beta^q+\beta=0\}$. Then $\Delta$ is an elementary abelian $p$-group of order $q^2$.
\begin{lemma}
\label{lemB310ct2016}
$ \Delta$ is an (elementary abelian) Sylow $p$-subgroup of $\aut_{\mathbb{K}}(K)$.
\label{nak}
\end{lemma}
\begin{proof} Let $S$ be a Sylow $p$-subgroup of $\aut_{\mathbb{K}}(K)$ containing $\Delta$. From Nakajima's bound \cite[Theorem 1]{N}, see also \cite[Theorem 11.84]{HKT},
$$|S| \leq \textstyle\frac{p}{p-2} (\gm(\cX)-1) = \textstyle\frac{p}{p-2} (q^2-2q) < pq^2,$$
whence $|S|=q^2$.
\end{proof}
\begin{rem} \emph{From the proof of Lemma \ref{nak}, if $q=p$ then $K$ hits the Nakajima's bound.}
\end{rem}
%Furthermore, let $W :=\{\varphi_{0,0,v} \mid v^m =1 \}$ and $V :=\{\varphi_{0,0,v} \mid v \in \mathcal{V}\}$.
\begin{lemma} \label{pro} The subgroups $\Delta$, $W$, $V$, $\Phi$ of $G$ have the following properties:
\begin{enumerate}
\item[\rm(i)] $\Delta$ is a normal subgroup of $G$.
\item[\rm(ii)]  $W$ is a subgroup of the center $Z(\Phi)$ of $\Phi$.
\item[\rm(iii)] $\Phi = \Delta \rtimes V$.
\item[\rm(iv)]  $G=\Delta\rtimes (V\rtimes \langle \xi \rangle)$.
\end{enumerate}
\end{lemma}
\begin{proof} By a  direct computation,
$$ \varphi_{\alpha_1,\beta_1,v_1}^{-1} \circ \varphi_{\alpha,\beta,1} \circ \varphi_{\alpha_1,\beta_1,v_1} = \varphi_{(\alpha_1 v_1^{-m} + \alpha)v_1^{-m} - \alpha_1v_1^{-m}, (\beta_1 v_1^{m} + \beta)v_1^{m} - \beta_1v_1^{m}, 1},$$
for every $\varphi_{\alpha_1,\beta_1,v_1} \in \Phi$ and $ \varphi_{\alpha,\beta,1} \in \Delta$. Also, $\xi\circ\varphi_{\alpha,\beta,1}\circ \xi=\varphi_{-\alpha,-\beta,1}$.
Therefore (i) holds. Furthermore, (ii) is proven by a straightforward computation.
Since $\Delta$ is a normal subgroup of $G$, and $|\Delta|$ is prime to $|V|$,  we have $\langle \Delta, V \rangle =\Delta V= \Delta \rtimes V$. Moreover, $|\Delta V|= |\Delta| |V|=|\Phi|$. Thus, $\Phi = \Delta \rtimes V$. From this, (iv) also follows.
\end{proof}
\begin{lemma} \label{short} The action of $\Delta$ on the set $\cP$ of places of $K$ has exactly two short orbits both of length $q$.
\end{lemma}
\begin{proof} From the Deuring-Shafarevich formula,
\begin{equation}
q^2-2q=\gamma(K) -1= |\Delta|(\gamma(K^\Delta) -1) + d, \label{ds}
\end{equation}
with $d=\sum_{i=1}^r (q^2- \lambda_i)$ where $\lambda_1 ,..., \lambda_r$ are the lengths of the $r$ short orbits of $\Delta$ in its action on $\cP$. Since $|\Delta|=q^2$, Equation (\ref{ds}) taken ${\rm{mod}}\,q^2$ yields that $d\geq q^2-2q$. Therefore, $\gamma(K^\Delta)=0$ and hence
$$q^2-2q= -q^2 + d. $$
Thus $i \leq 2$ and (\ref{ds}) reads $q^2-2q= -q^2+q^2-\lambda_1 +q^2 -\lambda_2 = q^2-(\lambda_1 +\lambda_2).$
whence $\lambda_1+\lambda_2=2q$, that is, $\lambda_1=\lambda_2=q$.
\end{proof}
For each point $P$ in a short orbit of $\Delta$, the fact that $\Delta$ is abelian together with Lemma \ref{short} yield the stabilizer $\Delta_P$ to have order $q$.
\begin{lemma}
\label{lem431luglio2016} For two points $P_1,P_2$ from different short orbits of $\Delta$, the stabilizers $\Delta_{P_1}$ and $\Delta_{P_2}$ have trivial intersection.
\end{lemma}
\begin{proof} By absurd, $\Delta_{P_1}$ fixes as many as $2q$ places of $K$. The Deuring-Shafarevich formula applied to $\Delta_{P_1}$ yields that $q=1-\bar{\gamma}$ where $\bar{\gamma}$ is
the $p$-rank of  $K^{\Delta_{P_1}}$. But this cannot actually occur as $q>2$.
\end{proof}
\begin{lemma} Let $\Omega$ be a short orbit of $\aut_{\mathbb{K}}(K)$ containing both short orbits of $\Delta$. Then $\Omega$ is the unique non-tame short orbit of $\aut_{\mathbb{K}}(K)$.
\end{lemma}
\begin{proof}
Take a place $P\in \cP$ outside $\Omega$. By absurd, the stabilizer of $P$ in $\aut_{\mathbb{K}}(K)$ contains a non-trivial $p$-subgroup. Let $S_p$ be a Sylow $p$-subgroup containing that subgroup. Lemma \ref{nak} together with claim (i) of Proposition \ref{pro} yields that $S_p=\Delta$. Now, the proof follows from Lemma \ref{short}.
\end{proof}
The following results provide characterizations of the short orbits of $\Delta$.
\begin{lemma} \label{09} $W$ fixes each place in the short orbits of $\Delta$.
\end{lemma}
\begin{proof} By Lemma \ref{pro}, $\Delta\times W$ is an abelian group. From Lemmas \ref{short} and \ref{lem431luglio2016}, $\Delta\times W$ induces a permutation group on both short orbits of $\Delta$. The nucleus of the permutation representation of $\Delta\times W$ on any of them has order $qm$ and hence it contains $W$, the unique subgroup of $\Delta\times W$ of order $m$.
\end{proof}
\begin{lemma} \label{zp} $\supp(\Div(s)_\infty)$ and $\supp(\Div(z)_\infty)$ are the short orbits of $\Delta$.
\end{lemma}
\begin{proof} From the proof of Proposition \ref{pro1}, the subfield $K^W$ is the Artin-Mumford function field $AM=\mathbb{K}(s,z)$ with (\ref{eq4}).  By (ii) of Lemma \ref{pro},
the centralizer of $W$ in $\aut_{\mathbb{K}}(K)$ contains $\Delta$. Since $W\cap \Delta=\{1\}$, the restriction of the action of  $\Delta$ on $AM$ is a subgroup of $\aut_{\mathbb{K}}(AM)$. On the other hand, $AM$ is the function field of the plane algebraic curve $\cC$ of affine equation $(X^q+X)(Y^q+Y)=1$ which has only two singular points, namely  $X_{\infty}$ and $Y_{\infty}$, both ordinary singularities of multiplicity $q$. On the set of places, that is, branches of $\cC$, $\Delta$ has a faithful action. Further, the unique Sylow $p$-subgroup $S_p$ of $\aut_{\mathbb{K}}(\cC)$ has order $q^2$ and a subgroup of $S_p$ of order $q$ fixes each of the $q$ places centered at $X_\infty$ and acts transitively on the set of the $q$ places centered at $Y_\infty$. Another subgroup of $S_p$ of order $q$ acts in the same way if the roles of the places centered at $X_\infty$ and $Y_\infty$ are interchanged. In particular, $\Delta=S_p$, and $\Delta$ has exactly two short orbits each of length $q$. In terms of $AM$,  $\Div(s)_\infty$ is the sum of the $q$ places centered at $X_\infty$.
%\textcolor{blue}{I am afraid this is not true for the AM function field. In fact, let $P_0,\ldots, P_{q-1}$ and $Q_0,\ldots, Q_{q-1}$ be the places centered at $X_\infty$  and $Y_\infty$ respectively. %Then $(s) = qQ_0-(P_0+\ldots+P_{q-1})$. The Lemma should be restated with $\supp(\Div(z)_\infty)$ instead of $\supp(\Div(s)_0)$.}
This together with Lemma \ref{09} shows that
the places of $M$ lying over these $q$ places in the extension $K|AM$ form a short orbit of $\Delta$.
Similarly, $\div(z)_\infty$ is the sum of the $q$ places centered at $X_\infty$, and the places of $K$ lying over the $q$ places centered at $Y_\infty$ form a short orbit of $\Delta$. From  Lemma \ref{short}, $\Div(s)_\infty$ and $\Div(z)_\infty$ are the short orbits of $\Delta$.
\end{proof}
 From now on $\Omega_1$ and $\Omega_2$ denote the two short orbits of $\Delta$ as given in Lemma \ref{short}. Up to a change of notation, $\Div(s)_0=\Omega_1$ and $\Div(s)_\infty=\Omega_2$.
 A byproduct of the proof of Lemma \ref{zp} is the following result.
\begin{lemma}
\label{lemA31oct2016} The stabilizer of any point $P\in \Omega_1$ in $\Delta$ consists of all $\varphi_{\alpha,0,1}$ with $\alpha^q+\alpha=0$. The same holds for $P\in \Omega_2$ and $\varphi_{0,\beta,1}$ with $\beta^q+\beta=0$.
\end{lemma}
We prove another result on the zeroes and poles of $x$.
\begin{lemma}
\label{lem1agosto2016} The zeroes of $x$, as well as the poles of $x$, have the same multiplicity.
\end{lemma}
\begin{proof} From Lemma \ref{zp}, any zero of $x$ is a point of $\Omega_1$. Since  $\Delta$ fixes $x$, and $\Omega_1$ is an orbit of $\Delta$, the claim follows for the zeroes of $x$.
The same argument works for the poles of $x$ whenever $\Omega_1$ is replaced by $\Omega_2$. Since $|\Omega_1|=|\Omega_2|$, we also have that the multiplicity of any zero of $x$ is equal to that of any pole of $x$.
\end{proof}
\begin{lemma} \label{raz} The subfield $K^\Delta$ of $K$ is rational.
\end{lemma}
\begin{proof} For a place $P\in \Omega_1\cup \Omega_2$, let $U$ be a subgroup of $\aut_{\mathbb{K}}(K)$ fixing $P$ whose order $u$ is prime to $p$. Then $U$ is a cyclic group. Suppose that $U$ centralizes $\Delta_P$. Then $U\Delta_P$ is an abelian group of order $uq$. Furthermore, the first $u+1$ ramification groups coincide, that is, $\Delta_P^{(0)}=\Delta_P^{(1)}=\ldots=\Delta_P^{(u)}$, see \cite[Lemma 11.75 (iv)]{HKT}. Since $\Delta_P=\Delta_P^{(0)}$ has order $q$ by Lemma \ref{short}, the Hurwitz genus formula applied to $\Delta$ gives
$$2\gg(K)-2\geq q(2\gg(K^\Delta)-2) +2q(q-1)(u+1)$$
By (ii) of Lemma \ref{pro} and Lemma \ref{09}, $U$ may be assumed to contain $W$. Then $2q^2(u+1)\geq 2q^2(m+1)$. This together with $2\gg(K)-2=2(q^2m-qm-q)$ yields $\gg(K^\Delta)=0$.
\end{proof}
The proof of Lemma \ref{raz} also gives the following result.
\begin{lemma}
\label{lem4Agosto2016} The centralizer of $\Delta$ in $\aut_{\mathbb{K}}(K)$ is $\Delta\times W$.
\end{lemma}
\section{Main result}
\label{mains}
Our goal is to prove the following result.
\begin{theorem}
\label{mr} Let $K$ be the Galois closure of the extension $F|H$ where $F=F(x,y)$ with $y^q+y=x^m+x^{-m}$, and $H=\mathbb{K}(x^{m(q-1)})$. Then
$\aut_{\mathbb{K}}(K)=\Delta\rtimes(C_{m(q-1)}\rtimes \langle \xi \rangle)$ where $\Delta$ is an elementary abelian normal subgroup of order $q^2$, $C_{(q-1)m}$ is a cyclic subgroup and $\xi$ is an involution.
\end{theorem}
In the proof we treat two cases separately depending upon the abstract structures of minimal normal subgroups of $\aut_{\mathbb{K}}(K)$.

\subsection{Case I: $\aut_{\mathbb{K}}(K)$ contains a solvable minimal normal subgroup}
\begin{lemma}
\label{lem22Agosto2016} If $N$ is a normal elementary abelian subgroup of $\aut_{\mathbb{K}}(K)$ of order prime to $p$ then either $N\leq W$ or $|N|\equiv |N\cap W|+1 \pmod p$.
\end{lemma}
\begin{proof} By (ii) of Lemma \ref{pro}, the conjugate of every element in $N\setminus N\cap W$ by any element of $\Delta$ is also in $ N\setminus N\cap W$. Assume on the contrary that $|N|-|N\cap W|\not\equiv1 \pmod p$. Then some element $u\in N\setminus N\cap W$ coincides with its own conjugate by any element of $\Delta$. Equivalently, $u$ centralizes $\Delta$.
By Lemma \ref{short}, $u$ preserves $\Omega_1$ (and $\Omega_2$). Since $u$ has prime order different from $p$,  $u$ fixes a place in $\Omega_1$. For $U=\langle u \rangle$, the argument used in the proof of Lemma \ref{raz} shows that $U$ is contained in $W$, a contradiction.
\end{proof}

Next, the possibility of the existence of some subgroup of $\aut_{\mathbb{K}}(K)$ which is not contained in $$G=\Phi \rtimes \langle \xi \rangle$$ is investigated.
\begin{lemma} \label{artm} Let $H$ be a subgroup of $\aut_{\mathbb{K}}(K)$ which is not contained in $G$. Then the centralizer of $H$ does not contain $W.$
\end{lemma}
\begin{proof} As already observed in the proof of Lemma \ref{zp}, the subfield $K^W$ is the Artin-Mumford function field $AM=\mathbb{K}(s,z)$ with (\ref{eq4}). By absurd, $HW/W$ is a subgroup of $\aut(AM)$.
Since $|\aut(AM)|=2(q-1)q^2$, see \cite[Theorem 7]{mv1982} for $q=p$ and \cite[Theorem 5.3]{KMam} for any $q$, and $G/W$ is a subgroup of $\aut(AM)$, the latter subgroup is the whole $\aut(AM)$. Therefore $HW/W$ is contained in $G/W$. But then $HW\leq G$ and hence $H\leq G$, a contradiction.
\end{proof}
From Proposition \ref{pro}, $M=\Delta\times W$. Therefore, $M$ is an abelian subgroup of $\Phi$ of order $q^2m$, and $|M|=q^2m>(q-1)(qm-1)=\gg(K).$
Let $R$ be the subgroup of $G$ generated by $M$ and $\xi$. Then  $R=M\rtimes \langle \xi \rangle$ as the normalizer of $M$ in $G$ contains $\xi$.
\begin{lemma}
\label{lem2agosto2016} If $N$ is an elementary abelian normal $2$-subgroup of $\aut_{\mathbb{K}}(K)$ then $N=\{1, \varphi_{0,0,-1}\}$.
\end{lemma}
\begin{proof}
%From the proof of Lemma \ref{lem16Bluglio2016}, $\xi\in G$. Also, $\varphi_{0,0,-1}\in G$.
By definition, $\xi$ and $\varphi_{0,0,-1}$ are contained in $G$.
Since both $\xi$ and $\varphi_{0,0,-1}$ are involutions and commute, they generate
an elementary abelian subgroup $S$ of $G$ of order $4$. Let $U$ be a subgroup of $\aut_{\mathbb{K}}(K)$ of order $d=2^u\geq 2$. From  the Hurwitz genus formula applied to $U$,
$$2\gg(K)-2=2^u(2\gg(K^U)-2)+\sum_{i=1}^k(2^u-\ell_i)$$
where $\ell_1,\ldots,\ell_k$ are the short orbits of $U$ on the set $\cP$ of all places of $K$. Since $\gg(K)=(q-1)(qm-1)$ is even, and hence $2\gg(K)-2\equiv 2 \pmod 4$, while $2^u(2\gg(K^U)-2)\equiv 0 \pmod 4$, some $\ell_i$ ($1\le i \le k$) must be either $1$ or $2$.  Therefore, $U$ or a subgroup of $U$ of index $2$ fixes a point of $\cX$ and hence
is cyclic. From \cite[Chapter I, Satz 14.9]{huppertI1967}, $U$ is
either cyclic, or the direct product of a cyclic group by a group of order $2$, or a generalized quaternion group, or dihedral, or semidihedral, or a modular maximal-cyclic group (also called type (3) with Huppert's notation). In particular, $U$ contains no elementary abelian subgroup of order $8$. By absurd, let $N$ be a elementary abelian normal $2$-subgroup of $\aut_{\mathbb{K}}(K)$ which is not contained in $G.$ Then $N$ has order $2$ or $4$. In the former case, $N$ is in $Z(\aut_{\mathbb{K}}(K))$ and hence $N$ together with $S$ generate an elementary abelian group of order $8$, a contradiction. If $|N|=4$
and $N\cap S=\{1\}$ then some non-trivial element of $s\in S$ commutes with each element of $N$, and hence $N$ together with $s$ generate an elementary abelian group of order $8$, again a contradiction. If $N\cap S=\{1,u\}$ then $u\in Z(G)$ and hence $u=\varphi_{0,0,-1}$. Since $|N|-|N\cap S|=2$, Lemma \ref{lem22Agosto2016} yields $N<W$ a contradiction. Therefore, $N<G$, and hence
$N$ is a subgroup of $V\rtimes \langle \xi \rangle$. Since $V$ is cyclic, $N$ contains $\varphi_{0,0,-1}$. If $|N|=4$ then $N$ has two elements outside $W$. But  this is impossible by Lemma \ref{lem22Agosto2016}.
\end{proof}
\begin{rem}
\label{rem5agosto2016} The proof of Lemma \ref{lem2agosto2016} also shows that $\aut_{\mathbb{K}}(K)$ contains no elementary abelian group of order $8$.
\end{rem}
\begin{lemma} \label{sol}
Any solvable minimal normal subgroup of $\aut_{\mathbb{K}}(K)$ is contained in $R$.
\end{lemma}
\begin{proof}  Let $N$ be a solvable minimal normal subgroup of $\aut_{\mathbb{K}}(K)$. Then $N$ is an elementary abelian group of order $r^h$ with a prime $r\geq 2$ and $h\ge 1$.
If $r=p$ then $N$ is contained in $\Delta$ by Lemma \ref{nak}. Therefore $r\neq p$ is assumed. By Lemma \ref{lem2agosto2016}, the case $r=2$ is dismissed, as well.

We investigate the subfield $K^N$. The quotient group $\hat{M}=MN/N$ is a subgroup of $\aut_{\mathbb{K}}(K^N)$. Since $p\neq r$, we have $\Delta\cap N=\{1\}$ and $M\cap N=W\cap N\leq W$. Furthermore, $\hat{M}\cong M/(M\cap N)\cong \Delta W/(W\cap N)$.
The Hurwitz genus formula applied to $N$ yields $\gg(K)-1 \geq |N|(\gg(K^N)-1)$.

We show that the $p$-rank $\gamma(K^N)$ of $K^N$ is positive. If  $\gamma(K^N)=0$ by absurd, any nontrivial $p$-subgroup of $\aut_{\mathbb{K}}(K^N)$ has exactly one fixed place, see \cite[Lemma 11.129]{HKT}.
Let $\hat{P}$ be the unique fixed place of $\hat{\Delta}=\Delta N/N$ viewed as a subgroup of  $\aut_{\mathbb{K}}(K^N)$. Then the $N$-orbit $o$ lying over $\hat{P}$ in the extension $K|K^N$ contains $\Omega_1\cup \Omega_2$. Furthermore, since $N$ is a normal subgroup of $\aut_{\mathbb{K}}(K)$, $o$ is the union of $\Delta$-orbits. By Lemma \ref{short} each $\Delta$-orbit other than $\Omega_1$ and $\Omega_2$ has size $q^2$. Therefore, $q$ divides $|o|$. Since $|o|$ divides $|N|$, this yields that $q$ divides $N$, a contradiction. As a consequence, $K^N$ is not rational.

We show that $K^N$ is neither elliptic. For a place $P\in \Omega_1$, all ramification groups $N_i^{(i)}$ of $N$ at $P$ have odd order, and hence $d_P=\sum_i (N_i^{(i)}-1)$ is even.
Let $\theta$ be the $N$-orbit containing $P$. Then, $|N_P||\theta|=|N|$. Take a Sylow $2$-subgroup $S$ of $G$ containing a Sylow $2$-subgroup $S_P$ of $G_P$. Since $\xi,\varphi_{0,0,-1}$ are two distinct involutions which commute, $S$ is not cyclic. Therefore $S\neq S_P$, as $S$ does not fix $P$. Thus $|S|$ does not divide $|G_P|$ showing that the $G$-orbit of $P$ must have even length. This yields that $\sum_{P\in \cP} d_P$ is divisible by four. On the other hand, $2\gg(K)-2=2(q^2m-qm-q)$ is twice an odd number, a contradiction.

Therefore, $\gg(K^N)\geq 2$.  From the Nakajima bound, see \cite{N}, or \cite[Theorem 11.84]{HKT} applied to $\hat{\Delta}$,

$$q^2 \leq \textstyle\frac{p}{p-2}(\gamma(K^N)-1)\leq \textstyle\frac{p}{p-2}(\gg(K^N)-1)$$
whence
\begin{equation}
\label{eq3Agosto2016}
\gg(K^N)-1 \geq
\begin{cases}
 {\mbox{$3$ \,\,\,when $q=3$}},\\
 {\mbox{$15$ when $q>3$}}.\\
\end{cases}
\end{equation}
From $|M| \geq \gg(K)-1$,
\begin{equation} \label{eql}
4|M| \geq 4(\gg(K)-1) \geq 4|N|(\gg(K^N)-1) = |N|(4\gg(K^N)+4-8)
\end{equation}
which  yields
\begin{equation}
4|M|\geq |N||\hat{M}| -8|N|.
\end{equation}
From $|N|(\gg(K^N)-1)\leq (\gg(K)-1)\leq |M|$,
\begin{equation}
\label{eq31agosto2016}
4 \geq \frac{|N|}{|M|} \frac{|M|}{|M \cap N|} -\frac{8|N|}{|M|} = \frac{|N|}{|M \cap N|} - \frac{8|N|}{|M|} \geq \frac{|N|}{|M \cap N|} -\frac{8}{\gg(K^N)-1}=\frac{|N|}{|W \cap N|} -\frac{8}{\gg(K^N)-1}.
\end{equation}
This and (\ref{eq3Agosto2016}) yield
\begin{equation} \label{ord}
\frac{|N|}{|W \cap N|}\leq
\begin{cases}
 {\mbox{$6$ when $q=3$}},\\
 {\mbox{$4$ when $q>3$}}.\\
\end{cases}
\end{equation}
Since $W \cap N \leq N$ we have $|W\cap N| =r^w$ for some $0\le w \leq h$. By (\ref{eq3Agosto2016}) and Lemma \ref{lem22Agosto2016}, this is only possible when either $r=3$ and $p\neq 3$, or $r=5$ and $q=3$, or $w=h$. In the latter case, $W\cap N=N$ whence $N\leq W< R$, and the claim is proven.
If $r=3$ and hence $|N|=3$ or $|N|=9$ according as $N\cap W=\{1\}$ or $|N\cap W|=3$,
Lemma \ref{lem22Agosto2016} shows that $N\leq W< R$. The same argument works for $r=5$, $|N|=5,25$, and $|N\cap W|=1,5$.
\end{proof}

\begin{lemma} \label{le1} If a normal subgroup $N$ of $\Phi$ is contained in $\Delta$ then $N$ coincides with $\Delta$.
\end{lemma}

\begin{proof} Take $\varphi_{\alpha,\beta,1} \in N$ for some $\alpha \ne 0$, (or  $\beta\ne 0$). Since $v$ has order $m(q-1)$ in $\mathbb{F}_{q^r}$. $v^m$ is a primitive element of $\mathbb{F}_{q}$.
Since $N$ is normal in $\aut_{\mathbb{K}}(K)$, $\varphi_{0,0,v}^{-1} \circ \varphi_{\alpha,\beta,1} \circ \varphi_{0,0,v} \in N$. From
$$ \varphi_{0,0,v}^{-1} \circ \varphi_{\alpha,\beta,1} \circ \varphi_{0,0,v}(x,s,z)=(x,s+v^m\alpha,z+v^{-m}\beta),$$
$\varphi_{0,0,v}^{-1} \circ \varphi_{\alpha,\beta,1} \circ \varphi_{0,0,v}=\varphi_{v^m\alpha,v^m\beta,1}$. Since $v^m$ is a primitive element of $\mathbb{F}_q$,
$N$ contains each $\varphi_{\alpha',\beta',1}$ whenever $\alpha'=\omega\alpha,\beta'=\omega^{-1}\beta$ with $\omega \in \mathbb{F}_q^{*}$. Thus $|N|\geq q$.
Moreover if $\alpha_i=\omega_i\alpha$ and $\beta_i=\omega_i^{-1}\beta$, where $\omega_i \in \mathbb{F}_q^{*}$ and $i=1,2$ then $N$ contains
$\varphi_{\alpha_1,\beta_1,1} \circ \varphi_{\alpha_2,\beta_2,1}=\varphi_{(\omega_1+\omega_2)\alpha,(\omega_1^{-1}+\omega_2^{-1})\beta,1}.$
To count the elements in $N$, observe that $(\omega + \omega')^{-1} = \omega^{-1}+ \omega'^{-1}$ only occurs whenever $\omega'$ is a root of the quadratic polynomial $\omega x+\omega^2+x^2$. For a fixed $\omega$, this shows that at least $(q-1)-2=q-3$ possible choices for $\omega'$ provide different elements in $N$. Thus, $|N| \geq q+(q-1)(q-3)=q^2-3(q-1).$

By absurd, $N$ is a proper subgroup of $\Delta$. Then
$ q^2-3(q-1) \leq \textstyle\frac{q^2}{p},$
which is only possible for $q=p=3$. In this case, since $\psi\circ \varphi_{\alpha,\beta,1}\circ \psi \in N$ we find $q-1$ more elements in $N$ of the form $\varphi_{\beta',\alpha',1}$, where $\alpha'=\omega \alpha$ and $\beta'=\omega \beta$ for $\omega \in \mathbb{F}_q^{*}$. Thus,
$|N| \geq q^2-2(q-1)=5.$
Since $\frac{q^2}{p}=3$ the claim also holds in this case.
\end{proof}
\begin{lemma} \label{le11} Let $N$ be a normal subgroup $M$ of $R$. If $|N|=r^h$, with an odd prime $r$ different from $p$, then $N$ is a subgroup of $W$.
\end{lemma}
\begin{proof} From $[R:M]=2$, $N$ is a subgroup of $M=\Delta\times W$. Since $N\cap \Delta=\{1\}$, this is only possible when $N<W$.
\end{proof}
Lemmas \ref{lem2agosto2016}, \ref{sol}, \ref{le1}, \ref{le11} have the following corollary.
\begin{lemma} \label{rp} Let $N$ be a solvable minimal normal subgroup of $\aut_{\mathbb{K}}(K)$. Then either
\begin{itemize}
\item[\rm{(i)}] $N=\Delta$, and $|N|=q^2$,
\item[\rm{(ii)}] $N<W$, and $|N|=r$ with a prime $r$ different from $p$.
\end{itemize}
\end{lemma}
\begin{lemma} \label{delta} If $\aut_{\mathbb{K}}(K)$ has a solvable minimal normal subgroup then $\Delta$ is a normal subgroup of $\aut_{\mathbb{K}}(K)$.
\end{lemma}
\begin{proof} We may assume that (ii) of Lemma \ref{rp} holds. Then $N=\langle \varphi_{0,0,w}\rangle$ with $w^r=1$. Therefore, the fixed places of $N$ are the zeroes and poles of $x$. From Lemma
\ref{zp}, these points form $\Omega_1\cup \Omega_2$. Hence $\aut_{\mathbb{K}}(K)$ preserves $\Omega_1\cup \Omega_2$. Therefore, the conjugate $\Delta'$ of $\Delta$ by any $h\in \aut_{\mathbb{K}}(K)$ has its two short orbits $\Omega_1'$ and $\Omega_2'$ contained in $\Omega_1\cup \Omega_2$. Actually, $\Omega_1'\cup\Omega_2'=\Omega_1\cup \Omega_2$. From this we infer that $\Delta=\Delta'$. Take any place $P\in \Omega_1$. Then
$|\Delta_P|=|\Delta_P'|=q$. Then both $\Delta_P$ and $\Delta_P'$ are contained in the unique $p$-subgroup $S_P$ of the stabilizer of $P$ in $\aut_{\mathbb{K}}(K)$, see \cite[(ii)a Theorem 11.49]{HKT}.
If $\Delta'_P \ne \Delta_P$  then $|S_P|>q$. Let $S$ be a Sylow $p$-subgroup of $\aut_{\mathbb{K}}(K)$. By Lemma \ref{nak}, $S$ is conjugate to $\Delta$ in $\aut_{\mathbb{K}}(K)$. But this is impossible as $|\Delta_Q|\leq q$ for any $Q\in \cX$ by Lemma \ref{short}. The same argument works for any place in $\Omega_2$. Since $\Delta_P$ and $\Delta_Q$, with $P\in \Omega_1,Q\in \Omega_2$, generate $\Delta$, it turns out
that $\Delta'$ is also generated by $\Delta_P$ and $\Delta_Q$. Thus $\Delta=\Delta'$.
\end{proof}
\begin{lemma} \label{W1} If $\aut_{\mathbb{K}}(K)$ has a solvable minimal normal subgroup then $W$ is a normal subgroup of $\aut_{\mathbb{K}}(K)$.
\end{lemma}
\begin{proof}  We may assume that (ii) of Lemma \ref{rp} holds. From Lemma \ref{lem4Agosto2016}, $\Delta\times W$ is a normal subgroup of $\aut_{\mathbb{K}}(K)$. Since $|\Delta|$ and $|W|$ are coprime,
the assertion follows.
\end{proof}

\begin{theorem} \label{aut1} If $\aut_{\mathbb{K}}(K)$ has a minimal normal subgroup which is solvable then $\aut_{\mathbb{K}}(K)=G$. In particular $|\aut_{\mathbb{K}}(K)| = 2q^2(q-1)m$.
\end{theorem}

\begin{proof} As usual, the factor group $\aut_{\mathbb{K}}(K)/\Delta$ is viewed as a subgroup of $\aut_{\mathbb{K}}(K^\Delta)$. Since $\xi$ interchanges $\Omega_1$ and $\Omega_2$, Lemma \ref{short} yields that
$\aut_{\mathbb{K}}(K^\Delta)$ has an orbit of length $2$ consisting of the points lying under $\Omega_1$ and $\Omega_2$ in the field extension $K|K^\Delta$. From Lemma \ref{raz}, $K^\Delta$ is rational. Hence $\aut_{\mathbb{K}}(K^\Delta)$ is isomorphic to a subgroup of $PGL(2,\mathbb{K})$. From the classification of subgroups of $PGL(2,\mathbb{K})$, $\aut_{\mathbb{K}}(K^\Delta)$ is a dihedral group. This shows that $\aut_{\mathbb{K}}(K)$ contains a (normal) subgroup $T$ of index $2$ such that $T = \Delta \rtimes C$ with a cyclic group $C$. Observe that $T$ is the subgroup of
$\aut_{\mathbb{K}}(K)$ which preserves both $\Omega_1$ and $\Omega_2$. Hence $W\le T$. From Lemma \ref{W1}, $CW$ is a group. Since its order $|C||W|/|C\cap W|$ is prime to $p$, this yields $W\leq C$.
Therefore, the assertion follows from Lemma \ref{artm}.
\end{proof}
\subsection{Case II: $\aut_{\mathbb{K}}(K)$ contains no solvable minimal normal subgroup}
For the rest of the paper we assume that  $\aut_{\mathbb{K}}(K)$ has no solvable minimal normal subgroup. In particular, $O(\aut_{\mathbb{K}}(K))$ is trivial, that is, $\aut_{\mathbb{K}}(K)$ is an odd-core free group.
Therefore, any minimal normal subgroup $N$ of $\aut_{\mathbb{K}}(K)$ is the direct product of pairwise isomorphic non-abelian simple groups. Since $\aut_{\mathbb{K}}(K)$ has no elementary abelian subgroup of order $8$, see the proof of Lemma \ref{lem2agosto2016},  this direct product has just one factor, that is, $N$ itself is a non-abelian simple group. The possibilities for $N$ are listed below.
\begin{enumerate}
\item[(I)] $N \cong PSL(2,\bar{q})$, where $\bar{q} \geq 5$ odd (The Sylow $2$-subgroups of $N$ is dihedral);
\item[(II)] $N \cong PSL(3,\bar{q})$, where $\bar{q} \equiv 3$ mod $4$ (The Sylow $2$-subgroups of $N$ are semidihedral);
\item[(III)] $N \cong PSU(3,\bar{q})$, where $\bar{q} \equiv 1$ mod $4$ (The Sylow $2$-subgroups of $N$ are semidihedral);
\item[(IV)] $N \cong \bA_7$ (The Sylow $2$-subgroups of $N$ are dihedral);
\item[(V)] $N \cong M_{11}$ (The Sylow $2$-subgroups of $N$ are semidihedral).
\end{enumerate}
%\end{lemma}

\begin{lemma}
\label{lem5agosto2016} If no minimal normal subgroup of $\aut_{\mathbb{K}}(K)$ is solvable, and $N$ is a non-abelian minimal normal subgroup of $\aut_{\mathbb{K}}(K)$, then $\Delta$ is contained in $N$.
\end{lemma}
\begin{proof} Since $N$ is a normal subgroup, its centralizer $C(N)$ in $\aut_{\mathbb{K}}(K)$ is also a normal subgroup of $\aut_{\mathbb{K}}(K)$. Actually $C(N)$ is trivial. In fact, on one hand,  $C$ has odd order, since an involution in $\aut_{\mathbb{K}}(K)$ together with an elementary abelian group of $N$ of order $4$ would generate an elementary abelian group of order $8$ contradicting the claim in Remark \ref{rem5agosto2016}. On the other hand, groups of odd order are solvable by the Feith-Thompson theorem.
By conjugation, every $d\in \Delta$ defines a permutation on $N$, and hence $\Delta$ has a permutation representation on $N$. Its kernel is contained in the centralizer $C(N)$, and hence is trivial, that is, the permutation representation is faithful. Therefore, $\Delta$ is isomorphic to a subgroup $D$ of the automorphism group $\aut(N)$ of $N$.

We show that $D\cap N\neq \{1\}$. By absurd, $\aut(N)/N$ contains the subgroup $DN/N\cong D$. Then case (I) does not occur since $\aut(N)\cong P\Gamma L(2,\bar{q})$ while the factor group $P\Gamma L(2,\bar{q})/PGL(2,\bar{q})$ is cyclic and $[PGL(2,\bar{q}):PSL(2,\bar{q})]=2$, and hence the odd order subgroups of $\aut(N)/N$ are all cyclic. In case (II),
$\aut(N)\cong P\Gamma L(3,\bar{q})$ while the factor group $P\Gamma L(3,\bar{q})/PGL(3,\bar{q})$ is cyclic and $[PGL(3,\bar{q}):PSL(2,\bar{q}):3]=1,3$ according as $q\equiv \pm 1 \pmod 3$.
Therefore, an odd order subgroup of $\aut(N)/N$ is an elementary abelian group of order $q^2$ only for $q=3$ and $\bar{q}\equiv 1 \pmod 3$. Furthermore, if $\bar{q}\equiv 1 \pmod 3$ then $|N|$ also divisible by $3$. Therefore, case (ii) does not occur either. Case (III) can be ruled out with the same argument replacing the condition $q\equiv \pm 1$ $\pmod 3$ with $q\equiv \mp 1$ $\pmod 3$.  In cases (IV), $|\aut(N)/N|=2$ and $|\aut(N)/N|=1$ respectively, and they contain no nontrivial subgroups of odd order.

The nontrivial subgroup $D\cap N$ is contained in a Sylow $p$-subgroup $S_p$ of $N$. Since $N$ is a normal subgroup of $\aut_{\mathbb{K}}(K)$, Lemma \ref{nak} yields that $D\cap N$ is a subgroup of $\Delta$.
Since $D\cap N$ is a normal subgroup of $G$, Lemma \ref{le1} shows that $D\cap N=\Delta$. Therefore, $\Delta<N$.
\end{proof}
\begin{proposition}
\label{prop6agosto2016} $\aut_{\mathbb{K}}(K)$ has a minimal normal solvable subgroup.
\end{proposition}
\begin{proof} By absurd, $\aut_{\mathbb{K}}(K)$ has no minimal solvable subgroup, and hence it has a minimal normal simple subgroup isomorphic to one of the five simple groups listed above.
From the proof of Lemma \ref{lem5agosto2016}, the centralizer of $N$ in $\aut_{\mathbb{K}}(K)$ is trivial. Therefore, we have a monomorphism $\tau:\aut_{\mathbb{K}}(K)\to \aut(N)$ defined by the map which takes $g\in \aut_{\mathbb{K}}(K)$ to the automorphism $\tau(g)$ of $N$ acting on $N$ by conjugation with $g$. Since $\tau$ maps $N$ into a normal subgroup $\tau(N)$  of $\aut(N)$ and $\Delta<N$ by Lemma \ref{lem5agosto2016}, we have that $\tau(N)$ has a subgroup isomorphic to $\Delta$.

In Case (I), $\tau(N)=PSL(2,\bar{q})$, and $\bar{q}=q^2$ by Lemma \ref{nak} and the classification of subgroups of $PSL(2,\bar{q})$. From Lemma \ref{lem4Agosto2016}, the centralizer of $\Delta$ in $\aut_{\mathbb{K}}(K)$ contains an element of order prime to $p$. Obviously, the same holds for $\tau(\Delta)$ where $\tau(\Delta)<\tau(N)\cong PSL(2,\bar{q})$. But this is impossible since $\aut(PSL(2,\bar{q}))=P\Gamma L(2,\bar{q})$ and any subgroup of $P\Gamma L(2,q)$ of order $\bar{q}$  coincides with its own centralizer in $P\Gamma L(2,q)$.

In Case (II), $\tau(N)= PSL(3,\bar{q})$ and $q$ must be a divisor of $\bar{q}-1$. The latter claim  follows from the fact that  $PSL(3,\bar{q})$ has order $\bar{q}^3(\bar{q}^2+\bar{q}+1)(\bar{q}+1)(\bar{q}-1)^2/\mu$ with $\mu=3,1$ according as $\mu\equiv \pm 1 \pmod 3$ where its subgroups of order $\bar{q}^3$ are not abelian while its subgroups of order $\bar{q}^2+\bar{q}+1$ and of order $\bar{q}+1$ are cyclic. Therefore, $\tau(\Delta)$ is
a Sylow subgroup contained in a subgroup which is the direct product of two cyclic groups of order $\bar{q}-1$. Since $\bar{q}-1$ is even, this shows that the centralizer of $\tau(\Delta)$ in $PSL(3,\bar{q})$ contains an elementary abelian subgroup of order $4$. Since $\tau$ is a monomorphism, the same holds for the centralizer of $\Delta$ in $\aut_{\mathbb{K}}(K)$. But this contradicts  Lemma \ref{lem4Agosto2016}.

Case (III) can be ruled out with the argument used for Case (II) whenever $\bar{q}-1$ and $\bar{q}+1$ are interchanged.

In Cases (IV) and (V), we have $\aut(N)=\bS_7$ and $\aut(N)=M_{11}$ respectively. The only Sylow subgroups of $N$ whose orders are square numbers have order $9$, and they coincide with their own centralizers in $\aut(N)$ contradicting Lemma \ref{lem4Agosto2016}.
\end{proof}

\section{Some Galois subcovers of $K$} \label{quot}
We investigate the possibility that some Galois subcovers of the Galois closure $K$ of $F|H$ are of the same type of $K$ with different defining pair $(q,m)$  of parameters. More precisely, we consider the family of all function fields $\bar{F}(x,y)$ with $y^{\bar{q}}+y=x^{\bar{m}}+x^{-\bar{m}}$ where $\bar{q}=p^k$, $\bar{m}$ is any positive integer prime to $p$, and find sufficient conditions on the parameters $\bar{q}$ and $\bar{m}$ ensuring that the Galois closure $\bar{K}$ of the extension $\bar{F}|\bar{H}$ be isomorphic to a subfield of $K^H$ for a subgroup $H$ of $\aut_{\mathbb{K}}(K)$.

First we point out that this can really occur.
\begin{proposition} For any divisor $d$ of $m$, let $C$ be the subgroup of $W$ of order $d$, and set $\bar{m}=m/d$. Then the subfield $K^C$ of $K$ is $\mathbb{K}(t,s,z)$ with (\ref{eq4}) and
\begin{equation}
\label{eq13} z^q+z=t^{\bar{m}},
\end{equation}
and $K^C$ is isomorphic to $\bar{F}$ for $\bar{q}=q$ and $\bar{m}$.
\end{proposition}
\begin{proof}
 The rational function $t=x^d$ is fixed by $C$. Since $[\mathbb{K}(K) : \mathbb{K}(K^C)]=d$ and $\mathbb{K}(t,s,z) \subset K^C$, the claim follows.
\end{proof}
Next we show examples with $\bar{q}<q$ arising from subfields of $\mathbb{F}_q$. For this purpose, we need a slightly different representation for $K$ and its $\mathbb{K}$-automorphism group. Take two nonzero elements $\mu,\theta\in \mathbb{K}$
 such that $\mu^q+\mu=0$ and $\theta^m=-\mu^{-1}$, and define
$x'=\theta^{-1}x,\,s'=\mu^{-1}s,\,z'=\mu z$. Then $K=\mathbb{K}(x',s',z')$ with
\begin{equation} \label{eq16}
s'^q-s'=\frac{1}{z'^q-z'},
\end{equation}
and
\begin{equation} \label{eq17}
z'^q-z'=x'^m,
\end{equation}
In fact, from (\ref{eq4}),
\begin{equation*} \label{eq14}
1=(s^q+s)(z^q+z) = (\mu^qs'^{q} + \mu s')(\mu^{-q}z'^{q}+\mu^{-1}z')=(s'^{q}-s')(z'^{q}-z'),
\end{equation*}
while, from (\ref{eq3}),
\begin{equation} \label{eq15}
-\mu^{-1}x'^m=x^m=z^q+z=-\mu^{-1}(z'^q-z').
\end{equation}
Let $\mathbb{F}_{q^r}$ the smallest Galois extension of $\mathbb{F}_q$ such that $m\mid (q^r-1)$.
For $\alpha',\beta'\in \mathbb{F}_q$ and $v^{m(q-1)}=1$ with $v\in \mathbb{F}_{q^r}$, let $$\varphi'_{\alpha',\beta',v'}(x',s',z')=(v'x',-v'^{-m}s'+\alpha',-v'^m z'+\beta').$$ Then $\varphi'_{\alpha',\beta',v'}(x',s',z')=\varphi_{\alpha,\beta,v}(x,s,z)$ for $\alpha=\alpha',\beta=\beta',v=v'\theta^{-1}$, and $\varphi'_{\alpha',\beta',v}(x',s',z')\in \aut_{\mathbb{K}}(K)$. Let
$$\Delta'= \{\varphi'_{\alpha',\beta',1} \mid \alpha', \beta' \in \mathbb{F}_q \},\qquad C_{m(q-1)}=\{\varphi'_{0,0,v}\mid v^{(m(q-1)}=1\},$$
and
$$\xi': (x',s',z') \mapsto \Big(\frac{1}{x'}, z', s' \Big).$$
Theorem \ref{mr} shows that $\Delta$ is the unique Sylow $p$-subgroup of $\aut_{\mathbb{K}}(K)$. Since $|\Delta'|=|\Delta|$, we have $\Delta'=\Delta$. Furthermore, Lemma \ref{zp} remains valid if $x$ is replaced by $x'$, and Lemma \ref{lemA31oct2016} remains valid, as well, if it is referred to $\varphi'_{\alpha',0,1}$ and $\varphi'_{0,\beta',1}$ with $\alpha',\beta'\in \mathbb{F}_q$.

Now, take  a nontrivial subfield $\mathbb{F}_{\bar{q}}$  of $\mathbb{F}_q$. Then
$$\tilde{\Delta}:=\{ \varphi'_{\alpha',\beta',1} \mid \alpha', \beta' \in \mathbb{F}_{\bar{q}} \},$$
is a subgroup of $\Delta$ of order $\bar{q}^2$. The subgroup $\tilde{\Delta}_1$ of $\tilde{\Delta}$ of order $\bar{q}$ consisting of all $\varphi'_{\alpha',0,1}$ with $\alpha'\in \mathbb{F}_{\bar{q}}$ fixes every place in $\Omega_1$, and the same holds for $\Omega_2$ when the subgroup $\tilde{\Delta}_2$ of all $\varphi'_{0,\beta',1}$ with $\beta'\in \mathbb{F}_{\bar{q}}$ is considered.

The following lemmas give basic information on the Galois subcover $K^{\tilde{\Delta}}$.
\begin{lemma} \label{gprank} The genus and $p$-rank of $K^{\tilde{\Delta}}$ are
$$\gg(K^{\tilde{\Delta}})=\Big(\frac{q}{\bar{q}}\,m-1\Big)\Big(\frac{q}{\bar{q}}-1\Big),\quad{\mbox{and}}\quad\gamma(K^{\tilde{\Delta}})=\Big(\frac{q}{\bar{q}}-1\Big)^2.$$
\end{lemma}
\begin{proof}
From Lemma \ref{raz}, $K^{\Delta}$ is rational and the different in the Hurwitz genus formula applied to $\Delta$ is
$$\sum_{P \in \cP} \sum_{i=0}^{m} (|{\Delta_P}^{(i)}| -1)=2\gg(K)-2+2q^2=2q(m+1)(q-1),$$ where $\cP$ is the set of all places of $K$.
On the other hand, $\Delta_P$ is nontrivial if and only if $P\in \Omega_1\cup \Omega_2$. From Lemma \ref{09}, $W\times \Delta_P$ fixes $P$, and hence for any $P\in \Omega_1\cup \Omega_2$ $q={\Delta_P}^{(0)}={\Delta_P}^{(1)}=\cdots ={\Delta_P}^{(m)}$; see \cite[Lemma 11.75 (i)]{HKT}. Also $|\Omega_1|+|\Omega_2|=2q$ and $|\Delta_P|=q$. Therefore, ${\Delta_P}^{(i)}$ is trivial for every $i >m$. By the properties of the subgroups $\tilde{\Delta}_1$ and $\tilde{\Delta}_2$, this yields for any point $P\in \Omega_1\cup \Omega_2$ that $\bar{q}={\tilde{\Delta}_P}^{(0)}={\tilde{\Delta}_P}^{(1)}=\cdots ={\tilde{\Delta}_P}^{(m)}$ but ${\tilde{\Delta}_P}^{(i)}$ is trivial for $i>m$. Therefore, the different in the Hurwitz genus formula applied to $\tilde{\Delta}$ is
$$\sum_{P \in \cP} \sum_{i=0}^{m} (|{\tilde{\Delta}_P}^{(i)}| -1)=2\gg(K)-2+2q^2=2q(m+1)(\bar{q}-1),$$
Thus,
$$2(qm-1)(q-1)-2=\bar{q}^2(2\gg(K^{\tilde{\Delta}}) -2) +2q(m+1)(\bar{q}-1),$$
whence the first claim follows. Moreover, from the Deuring-Shafarevic formula applied to $\tilde{\Delta}$,
$$(q-1)^2-1=\bar{q}^2(\gamma(K^{\tilde{\Delta}})-1)+2\frac{q}{\bar{q}}(\bar{q}^2-\bar{q}),$$
whence the second claim follows.
\end{proof}
\begin{proposition} \label{qk} Let $q=\bar{q}^k$ with $k \geq 1$. Then $K^{\tilde{\Delta}}=\mathbb{K}(x',t,w)$ with
$$\left \{
\begin{array}{lll}
w+w^{\bar{q}}+\ldots+w^{{\bar{q}}^{k-1}}=x'^{-m},\\
t+t^{\bar{q}}+\ldots+t^{{\bar{q}}^{k-1}}=x'^m.
\end{array}
\right.
$$
Furthermore, $\aut_{\mathbb{K}}(K^{\tilde{\Delta}})$ has a subgroup $\bar{G}$ of order $2(q/\bar{q})^2m(\bar{q}-1)$ with
$\bar{G}=(\Delta/\tilde{\Delta})\rtimes(C_{m(\bar{q}-1)}\rtimes \langle \bar{\xi} \rangle)$.
\end{proposition}
\begin{proof} First we show that $K^{\tilde{\Delta}}=\mathbb{K}(x',t,w)$ with $t=s'^{\bar{q}}-s',$ and $w=z'^{\bar{q}}-z'.$
By direct computation, both $t$ and $x'$ are fixed by $\tilde{\Delta}$. Hence $\mathbb{K}(x',t,w) \subseteq K^{\tilde{\Delta}}$. Also $[K:K^{\tilde{\Delta}}]=\bar{q}^2$.
On the other hand, both extensions $K|\mathbb{K}(x',s,w)$ and $\mathbb{K}(x',s,w)|\mathbb{K}(x',t,w)$ are (Artin-Schreier extensions) of degree $\bar{q}$,
$$[K: \mathbb{K}(x',t,w)] =[K: \mathbb{K}(x',s,w)]\cdot [\mathbb{K}(x',s,w): \mathbb{K}(x',t,w)]=\bar{q} \cdot\bar{q}=\bar{q}^2.$$
Therefore $K^{\tilde{\Delta}}=\mathbb{K}(x',t,w)$.
Since
$z'^q-z'=z'^{\bar{q}^k}-z'^{\bar{q}^{k-1}}+z'^{\bar{q}^{k-1}}- \ldots +z'^{\bar{q}}-z'^{\bar{q}} -z'=\sum_{i=0}^{k-1}w^{\bar{q}^i},$
and this remains true when $z'$ and $w$ are replaced by $s'$ and $t$,
%$$s'^q-s'=s'^{\bar{q}^k}-s'^{\bar{q}^{k-1}}+s'^{\bar{q}^{k-1}}- ... +s'^{q'}-s'^{\bar{q}} -s'=\sum_{i=0}^{k-1}t^{q'^i},$$
the first claim follows. The second claim can be deduced from $\aut_{\mathbb{K}}(K)$ taking for $\bar{G}$ the normalizer of $\tilde{\Delta}$. Alternatively, a direct computation shows that
the following maps are elements of $\aut_{\mathbb{K}}(K^{\tilde{\Delta}})$:
$\bar{\varphi}_{\alpha,\beta,v}(x',t,w)=(vx',v^{-m}t+\alpha',v^mw+\beta)$ with $Tr_{\mathbb{F}_q|\mathbb{F}_{\bar{q}}}(\alpha)=Tr_{\mathbb{F}_q|\mathbb{F}_{\bar{q}}}(\beta)=0$, and $v^{m(\bar{q}-1)}=1$, and $\bar{\xi}(x',t,w)=(x'^{-1},w,t)$. These generate a group $\bar{G}$ with the properties in the second claim.
\end{proof}
\begin{corollary}
\label{corB1nov2016} If $q=\bar{q}^2$ then $K^{\tilde{\Delta}}$ is isomorphic to $F$ with parameters $(\bar{q},m)$.
\end{corollary}

From Lemma \ref{gprank}, for every $q=\bar{q}^k$ with  $k \geq 1$, $K^{\tilde{\Delta}}$ has the same genus and $p$-rank of the function field $F$ with parameters $(q/ \bar{q},m)$. Moreover, from Proposition \ref{qk}, $K^{\tilde{\Delta}}=\mathbb{K}(x',t,w)$ with
$$\left \{
\begin{array}{lll}
(w+w^{\bar{q}}+\ldots+w^{{\bar{q}}^{k-1}})(t+t^{\bar{q}}+\ldots+t^{{\bar{q}}^{k-1}})=1,\\
t+t^{\bar{q}}+\ldots+t^{{\bar{q}}^{k-1}}=x'^m,
\end{array}
\right. $$
and $\aut_{\mathbb{K}}(K^{\tilde{\Delta}})$ has a subgroup $\bar{G}$ of order $2(q/\bar{q})^2m(\bar{q}-1)$ with
$$\bar{G}=(\Delta/\tilde{\Delta})\rtimes(C_{m(\bar{q}-1)}\rtimes \langle \bar{\xi} \rangle)=
(\Delta/\tilde{\Delta}\rtimes(C_{m(\bar{q}-1)})\rtimes \langle \bar{\xi} \rangle $$
where the subgroup $\bar{W}=\Delta/\tilde{\Delta}\rtimes C_{m(\bar{q}-1)}$ consists of all maps $\bar{\varphi}_{\alpha,\beta,v}(x',t,w)=(vx',v^{-m}t+\alpha,v^mw+\beta)$ with $Tr_{\mathbb{F}_q|\mathbb{F}_{\bar{q}}}(\alpha)=Tr_{\mathbb{F}_q|\mathbb{F}_{\bar{q}}}(\beta)=0$, $v^{m(\bar{q}-1)}=1$, whereas $\bar{\xi}(x',t,w)=(x'^{-1},w,t)$. In particular, the subgroup $C_m$ consisting of all
maps $\bar{\varphi}_{0,0,v}$ with $v^m=1$ is the center $Z(\bar{W})$ of $\bar{W}$, and $C_m$ is a normal subgroup of $\bar{G}$.

By Corollary \ref{corB1nov2016}, if  $q=\bar{q}^2$ then $K^{\tilde{\Delta}}$ and $F$ with parameter $(q/\bar{q},m)$ are isomorphic. Our aim is to prove that the converse also holds.
%\begin{proposition}
%\label{prop16nov2016} If $K^{\tilde{\Delta}}$ is isomorphic to $F$ with parameters $(q/\bar{q},m)$ then $q=\bar{q}^2$.
%\end{proposition}
%\begin{proof}

For this purpose, it is useful to view $\mathbb{K}(x',t,w)$ as a degree $m$ Kummer extension of the function field $L=\mathbb{K}(t,w)$ where $(w+w^{\bar{q}}+\ldots+w^{{\bar{q}}^{k-1}})(t+t^{\bar{q}}+\ldots+t^{{\bar{q}}^{k-1}})=1$. Since $L$ is the fixed field of $C_m$, and $C_m$ is a normal subgroup of $\bar{G}$, the factor group $\bar{G} / C_m$
is a subgroup of $\aut(L)$.
%\cong \Delta / \bar{\Delta} \rtimes (C_{\bar{q}-1} \rtimes C_2)$ is a subgroup of $\aut_\mathbb{K}(L)$. Here, $\Delta / \bar{\Delta} =
By direct computation, $\bar{G}/C_m$ contains the subgroup $\Delta^*$ consisting all maps $\bar{\varphi}_{\alpha,\beta}(t,w)=(t+\alpha,w+\beta)$ with $Tr_{\mathbb{F}_q|\mathbb{F}_{\bar{q}}}(\alpha)=Tr_{\mathbb{F}_q|\mathbb{F}_{\bar{q}}}(\beta)=0$ as well as the involution $\xi^*(t,w)=(w,t)$, and the subgroup $C_{\bar{q}-1}=\{\eta^*(t,w)=(\lambda t, \lambda^{-1} w)|\lambda^{\bar{q}-1}=1\}$. Therefore, $\bar{G}/C_m\cong (\Delta^*\rtimes C_{\bar{q}-1})\rtimes \langle \xi^*\rangle.$ Furthermore, $\Delta^*$ has two short orbits $\Omega_1^*$ and $\Omega_2^*$, the former consisting of all places centered at the infinite point $W_\infty$  of the curve $(W+W^{\bar{q}}+\ldots+W^{{\bar{q}}^{k-1}})(T+T^{\bar{q}}+\ldots+T^{{\bar{q}}^{k-1}})=1$, the latter one of those
centered at the other infinite point $T_\infty$. Both points at infinity are ordinary singular points with multiplicity $q/\bar{q}$. Now look at $\mathbb{K}(t,w)|\mathbb{K}(t)$ as a generalized Artin-Schreier extension of degree $q/ \bar{q}$. Then the (unique) zero of $t$ is totally ramified while each pole of $t$ is totally unramified. More precisely, $\Div(t)_0=(q/\bar{q}) P,$ while $\Div(t)_\infty = \sum_{i=1}^{q/\bar{q}}T_i$ with $\Omega^*_1=\{T_1,\ldots,T_{q/\bar{q}}\}$
where $P$ is the place corresponding to the unique branch centered at $W_\infty$ whose tangent has equation $T=0$.
%Similarly for $w$ and  $\Omega^*_2=\{W_1,\ldots,W_{q/\bar{q}}\}$.
By a direct computation, $C_{\bar{q}-1}$ fixes $P$ and acts transitively on the remaining $q/\bar{q}-1$ places in $\Omega^*_1$. Analogous results hold for $w$ and $\Omega^*_2$. Hence $C_{\bar{q}-1}$ fixes a unique point in $\Omega^*_2$ and acts transitively on the remaining $q/\bar{q}-1$ places in $\Omega^*_2$.
%\end{proof}
\begin{lemma} \label{genAM} Let $C \leq \aut_\mathbb{K}(L)$ be a cyclic group containing $C_{\bar{q}-1}$. If $C$ is in the normalizer $N_{\aut_\mathbb{K}(L)}(\Delta^*)$ and  leaves both short orbits of $\Delta^*$ invariant, then $C=C_{\bar{q}-1}$.
\end{lemma}

\begin{proof}
Let $C=\langle c \rangle$. Then $c$ preserves both $\Omega^*_1$. Since $c$ commutes with $C_{\bar{q}-1}$, it fixes $P$. Thus $t$ and the image $c(t)$ of $t$ by $c$ have the same poles and the same zero. Therefore, $c(t)=\rho t$ with some $\rho \in \mathbb{K}^*$. Analogously, $c(w)= \sigma w$ with some $\sigma\in \mathbb{K}^*$. By a straightforward computation, this yields $\rho=\sigma$ and $\rho^{\bar{q}-1}=1$.  Hence $c$ has order at most $\bar{q}-1$ and the claim holds.
\end{proof}

\begin{corollary} Let $q=\bar{q}^k$. Then $k\leq 2$ is the necessary and sufficient condition for $K^{\tilde{\Delta}}$ to be isomorphic to $F$ with parameter $(q / \bar{q},m)$.
\end{corollary}
\begin{proof} By Corollary \ref{corB1nov2016} we only have to prove the necessary condition.
 By absurd, $K^{\tilde{\Delta}}$ and $F$ with parameter $(q / \bar{q},m)$ have isomorphic $\mathbb{K}$-automorphism groups. From Theorem \ref{mr},  $\aut_{\mathbb{K}}(K^{\tilde{\Delta}})$ has a cyclic group of order $q/\bar{q}-1$ contained in the normalizer of $\tilde{\Delta}$. From the discussion after Corollary \ref{corB1nov2016}, this yields the  existence of a cyclic group $C$ of the same order $q/\bar{q}-1$ satisfying the hypotheses in Lemma \ref{genAM}. Therefore, $q/\bar{q}-1\le \bar{q}-1$ whence $k\leq 2$.
 \end{proof}

\begin{rem} {\em{From Corollary \ref{corB1nov2016}, a tower $\mathbb{K}(x)\subset F_1 \subset \cdots \subset F_i \subset \ldots$ arises where $q=p^{2^i}$ and $F_i$ is a function field isomorphic to $\mathbb{K}(x,y,z)$ defined by $y^{q}+y=x^m+x^{-m}$ and $z^q+z=x^m$. By Theorem \ref{mains},
$$\lim_{i=\infty}=\frac{|\aut_{\mathbb{K}}(F_i)|}{\gg(F_i)^{3/2}}=\frac{2}{\sqrt m}. $$}}
\end{rem}

    \end{document}